%% file: covering_minima_volume.tex
\documentclass[11pt,a4paper]{amsart}

\usepackage{epsfig}
\usepackage{graphicx}
\usepackage[usenames,dvipsnames]{xcolor}
\usepackage[colorlinks=true,linkcolor=Blue,citecolor=Green,urlcolor=Black]{hyperref}
\usepackage{amsmath,amsthm,amsfonts,amssymb}
\usepackage{nicefrac}
\usepackage[T1]{fontenc}
\usepackage[utf8]{inputenc}
\usepackage[textwidth=2.5cm,textsize=small]{todonotes}
\usepackage[capitalize]{cleveref}   
\crefname{enumi}{}{}
\usepackage[alwaysadjust]{paralist}
\usepackage{mathtools}
\usepackage{array}
\usepackage[belowskip=6pt,width=.95\textwidth]{caption}

\newtheorem{theorem}{Theorem}[section]
\newtheorem{lemma}[theorem]{Lemma}
\newtheorem{proposition}[theorem]{Proposition}
\newtheorem{corollary}[theorem]{Corollary}

\newtheorem{problem}[theorem]{Problem}

\newtheorem{remark}[theorem]{Remark}
\newtheorem{conjecture}[theorem]{Conjecture}

\DeclareMathOperator{\inter}{int}

\DeclareMathOperator{\conv}{conv}

\def\A{\mathcal{A}}
\def\K{\mathcal{K}}
\def\D{\mathcal{D}}

\def\L{\mathcal{L}}

\def\R{\mathbb{R}}

\def\Z{\mathbb{Z}}

\def\N{\mathbb{N}}

\def\M{\mathrm{M}}

\DeclareMathOperator{\vol}{vol}

\DeclareMathOperator{\lin}{lin}

\DeclareMathOperator{\GL}{GL}

\DeclareMathOperator{\LG}{LG}
\DeclareMathOperator{\diam}{diam}
\DeclareMathOperator{\dist}{d}
\DeclareMathOperator{\flt}{Flt}
\DeclarePairedDelimiter{\card}{\lvert}{\rvert}

\begin{document}

\title[On densities of certain lattice arrangements]{On densities of lattice arrangements intersecting every $i$-dimensional affine subspace}

\author{Bernardo Gonz\'{a}lez Merino}
\address{Zentrum Mathematik, Technische Universit\"at M\"unchen, Boltzmannstr. 3, 85747 Garching bei M\"unchen, Germany}
\email{bg.merino@tum.de}

\author{Matthias Henze}
\address{Institut f\"ur Mathematik, Freie Universit\"at Berlin, Arnimallee 2, 14195 Berlin, Germany}
\email{matthias.henze@fu-berlin.de}

\thanks{Parts of this research were carried out during the authors' stay at the Mathematisches Forschungsinstitut Oberwolfach within the program "Research in Pairs" in 2014.
The first author was partially supported by Fundaci\'{o}n S\'{e}neca, Science and Technology Agency of the Regi\'{o}n de Murcia, through the
Programa de Formaci\'on Postdoctoral de Personal Investigador and
the Programme in Support of Excellence Groups of the Regi\'{o}n de Murcia, Spain, project reference 19901/GERM/15, and MINECO project reference MTM2015-63699-P, Spain.
The second author was supported by the Freie Universität Berlin within the Excellence Initiative of the
German Research Foundation.}



\begin{abstract}
In 1978, Makai Jr.~established a remarkable connection between the volume-product of a convex body, its maximal lattice packing density and the minimal density of a lattice arrangement of its polar body intersecting every affine hyperplane.
Consequently, he formulated a conjecture that can be seen as a dual analog of Minkowski's fundamental theorem, and which is strongly linked to the well-known Mahler-conjecture.

Based on the covering minima of Kannan \& Lov\'{a}sz and a problem posed by Fejes T\'{o}th, we arrange Makai Jr.'s conjecture into a wider context and investigate densities of lattice arrangements of convex bodies intersecting every $i$-dimensional affine subspace.
Then it becomes natural also to formulate and study a dual analog to Minkowski's second fundamental theorem.
As our main results, we derive meaningful asymptotic lower bounds for the densities of such arrangements, and furthermore, we solve the problems exactly for the special, yet important, class of unconditional convex bodies.
\end{abstract}

\maketitle

\input{sec-intro}
\input{sec-discussion}
\input{sec-mu_i_vol_bounds}
\input{sec-product_functional}

\subsection*{Acknowledgments}
We thank Iskander Aliev for pointing us to the connection of covering radii of simplices and diameters of quotient lattice graphs.
We are grateful to Peter Gritzmann, Martin Henk, and Mar\'{i}a A.~Hern\'{a}ndez Cifre for providing opportunities for mutual research visits of the authors and for useful discussions on the topics of this paper.

\bibliographystyle{amsplain}
\bibliography{mybib}

\end{document}

%% file: sec-intro.tex
\section{Introduction}

A \emph{convex body} is a compact convex full-dimensional set in the Euclidean space $\R^n$.
We denote by~$\K^n$ the family of all convex bodies in $\R^n$, and we write $\K^n_o$ for the subfamily of \emph{$o$-symmetric} convex bodies $K\in\K^n$, that is, $K=-K$.
Let $\L^n$ denote the set of full-dimensional lattices in $\R^n$, that is, discrete subgroups of~$\R^n$ of full rank.
Every lattice $\Lambda\in\L^n$ can be written as $\Lambda=A\Z^n$ for some invertible matrix $A\in\GL_n(\R)$.
Let $\A_i(\R^n)$ be the family of $i$-dimensional affine subspaces of $\R^n$.
Given a subset $S\subseteq\R^n$ and a lattice $\Lambda\in\L^n$, we call $S+\Lambda=\bigcup_{z\in\Lambda}(S+z)$ a \emph{lattice arrangement}.
For more basic notation and background information on convex bodies and lattices, we refer to the textbooks by Gruber~\cite{gruber2007convex} and Martinet~\cite{martinet2003perfect}, respectively.

Our main interest in this paper concerns the so-called covering minima and their relation to the volume of a convex body.
Motivated by a number of applications, for instance, to flatness theorems and to inhomogeneous simultaneous diophantine approximation, these numbers have been formally introduced by Kannan \& Lov\'{a}sz~\cite{kannanlovasz1988covering}.
For a convex body $K\in\K^n$, a lattice $\Lambda\in\L^n$, and $i=1,\ldots,n$, the \emph{$i$-th covering minimum} (of $K$ with respect to~$\Lambda$) is defined as
\[\mu_i(K,\Lambda)=\inf\left\{\mu>0:\left(\mu K+\Lambda\right)\cap L\neq\emptyset\textrm{ for all }L\in\A_{n-i}(\R^n)\right\}.\]
Since $K$ is compact, this infimum is attained and thus can be replaced by a minimum.
These numbers extend the classical notion of the covering radius $\mu_n(K,\Lambda)=\min\{\mu>0:\mu K+\Lambda=\R^n\}$, which is also known as the inhomogeneous minimum (see~\cite[\S 13]{gruberlekker1987geometry}).
The concept of lattice arrangements that intersect every $(n\!-\!i)$-dimensional affine subspace has been studied already before the work of Kannan \& Lov\'{a}sz.
In fact, generalizing the sphere covering problem, Fejes T\'{o}th~\cite{toth1976research} proposed to determine the density of the thinnest lattice arrangement of solid spheres with this property (more on this in \cref{sectDiscussion}).

A basic result on \emph{lattice coverings}, meaning lattice arrangements covering the whole space, is that their density is at least one.
More precisely, for any $K\in\K^n$ and $\Lambda\in\L^n$, the \emph{density} of $K+\Lambda$ is defined as
\[\delta(K,\Lambda)=\frac{\vol(K)}{\det(\Lambda)},\]
where $\vol(K)$ denotes the volume (Lebesgue-measure) of~$K$ and $\det(\Lambda)=\card{\det(A)}$ the determinant of the lattice~$\Lambda=A\Z^n$.
Now, if $K+\Lambda=\R^n$, then $\delta(K,\Lambda)\geq1$, which can be equivalently formulated in the language of the covering radius as (cf.~\cite[\S 13.5]{gruberlekker1987geometry})
\begin{align}
\mu_n(K,\Lambda)^n\vol(K)&\geq\det(\Lambda).\label{eqn_mu_n}
\end{align}
The corresponding result for \emph{lattice packings}, that is, lattice arrangements $K+\Lambda$ with the property that $(\inter(K)+x)\cap(\inter(K)+y)=\emptyset$, for every $x,y\in\Lambda$, $x\neq y$, is the intuitively clear statement that their density is $\delta(K,\Lambda)\leq1$.
In order to give an equivalent formulation that corresponds to~\eqref{eqn_mu_n}, we introduce Minkowski's \emph{successive minima}, which are defined for any $o$-symmetric $K\in\K^n_o$ and any $\Lambda\in\L^n$ as
\[\lambda_i(K,\Lambda)=\min\{\lambda\geq0:\dim(\lambda K\cap\Lambda)\geq i\},\quad\text{for }i=1,\ldots,n.\]
Here, $\dim(S)$ denotes the dimension of the affine hull of the set $S\subseteq\R^n$.
For general $K\in\K^n$, one usually extends this definition by setting $\lambda_i(K,\Lambda):=\lambda_i(\frac12\D K,\Lambda)$, where $\D K=K-K$ is the \emph{difference body} of~$K$.
Based on the fact that $K+\Lambda$ is a lattice packing if and only if $\lambda_1(K,\Lambda)\geq2$ (cf.~\cite[Sect.~30]{gruber2007convex}), one can reformulate the density statement for lattice packings as
\begin{align}
\lambda_1(K,\Lambda)^n\vol(K)&\leq2^n\det(\Lambda).\label{eqnMinkowskiFirst}
\end{align}
In the case that $K$ is $o$-symmetric, this inequality is known as Minkowski's first fundamental theorem (cf.~\cite[Sect.~22]{gruber2007convex} and~\cite[\S 30]{minkowski1896geometrie}).
Minkowski strengthened his fundamental theorem by taking the whole sequence of successive minima into account.
He formulated his result for $o$-symmetric convex bodies, but analogously to~\eqref{eqnMinkowskiFirst} it naturally extends to arbitrary $K\in\K^n$ and $\Lambda\in\L^n$, and reads as follows (cf.~\cite[Sect.~23]{gruber2007convex} and~\cite{malikiosis2010adiscrete}):
\begin{align}
\lambda_1(K,\Lambda)\cdot\ldots\cdot\lambda_n(K,\Lambda)\vol(K)&\leq2^n\det(\Lambda).\label{eqnMinkowskiSecond}
\end{align}

The main motivation for our studies is a conjecture of Makai Jr.~which is an analog of~\eqref{eqn_mu_n} for the first covering minimum and at the same time a polar version of Minkowski's theorem~\eqref{eqnMinkowskiFirst}.

\begin{conjecture}[Makai Jr.~\cite{makai1978on}]\label{conjMakai}
Let $K\in\K^n$ and let $\Lambda\in\L^n$.
Then,
\begin{align}
\mu_1(K,\Lambda)^n\vol(K)&\geq\frac{n+1}{2^nn!}\det(\Lambda),\label{eqnMakaiConjGeneral}
\end{align}
and equality can only hold if $K$ is a simplex.

Moreover, if $K$ is $o$-symmetric, then
\begin{align}
\mu_1(K,\Lambda)^n\vol(K)&\geq\frac1{n!}\det(\Lambda),\label{eqnMakaiConjSymmetric}
\end{align}
and equality can only hold if $K$ is a crosspolytope.
\end{conjecture}

Makai Jr.~showed that for the standard lattice $\Lambda=\Z^n$, equality in~\eqref{eqnMakaiConjGeneral} holds for the simplex $T_n=\conv\{e_1,\ldots,e_n,-{\bf 1}\}$ and in~\eqref{eqnMakaiConjSymmetric} for the crosspolytope $C_n^\star=\conv\{\pm e_1,\ldots,\pm e_n\}$, where $e_i$ is the $i$-th coordinate unit vector and ${\bf 1}=(1,\ldots,1)^\intercal$ is the \emph{all-one vector}.

Note that Makai Jr.~did not state his conjecture in terms of $\mu_1(K,\Lambda)$ but rather in terms of what he calls \emph{non-separable arrangements}, which are lattice arrangements that intersect every affine hyperplane.
In order to explain why Makai Jr.'s conjecture would be a polar Minkowski theorem, we make use of an identity of Kannan \& Lov\'{a}sz~\cite[Lem.~(2.3)]{kannan1992lattice} that says that the first successive minimum is strongly dual to the first covering minimum.
More precisely, for any $K\in\K^n_o$ and $\Lambda\in\L^n$, we have
\begin{align}
\lambda_1(K,\Lambda)\mu_1(K^\star,\Lambda^\star)&=\frac12,\label{eqnDualitySuc1Cov1}
\end{align}
where $K^\star=\{x\in\R^n:x^\intercal y\leq 1,\text{ for all }y\in K\}$ is the \emph{polar body} of~$K$, and $\Lambda^\star=\{x\in\R^n:x^\intercal y\in\Z,\text{ for all }y\in\Lambda\}$ is the \emph{polar lattice} of~$\Lambda$.
By the definition of the first successive minimum, we have for any $K\in\K^n_o$,
\[\lambda_1(K,\Lambda)\geq1\quad\text{ if and only if }\quad\inter(K)\cap\Lambda=\{0\}.\]
Therefore, based on~\eqref{eqnDualitySuc1Cov1} we see that under the condition $\inter(K)\cap\Lambda=\{0\}$, Minkowski's theorem~\eqref{eqnMinkowskiFirst} states that $\vol(K)\leq2^n\det(\Lambda)$, whereas the $o$-symmetric part of \cref{conjMakai} claims that
\begin{align}
\vol(K^\star)&\geq\frac{2^n}{n!}\det(\Lambda^\star),\quad\text{for any }K\in\K^n_o\text{ with }\inter(K)\cap\Lambda=\{0\}.\label{eqnPolarMinkowski}
\end{align}

Yet another interpretation of Makai Jr.'s conjecture can be given in terms of the \emph{lattice width} $\omega_\Lambda(K)=\min_{v\in\Lambda^\star\setminus\{0\}}\omega(K,v)$ of $K$ with respect to $\Lambda$, where $\omega(K,v)=\max_{x\in K}x^\intercal v-\min_{x\in K}x^\intercal v$ is the width of $K$ in direction~$v$.
The identity~\eqref{eqnDualitySuc1Cov1} shows that $\mu_1(K,\Lambda)$ is reciprocal to $\omega_\Lambda(K)$ and hence~\eqref{eqnMakaiConjSymmetric} can be seen as a discrete analog to the obvious inequality $\vol(K)\geq\frac{\kappa_n}{2^n}\omega(K)^n$, for $K\in\K^n_o$, where $\omega(K)=\min_{u\in S^{n-1}}\omega(K,u)$ is the usual width of $K$ and $\kappa_n=\pi^{n/2}/\Gamma(n/2+1)$ is the volume of the Euclidean unit ball~$B_n$.

Rather than attacking \cref{conjMakai} directly, our main objective is to embed Makai Jr.'s problem into a wider context that strengthens the analogy to the covering inequality~\eqref{eqn_mu_n} and the duality to Minkowski's classical results.
In the spirit of Fejes T\'{o}th~\cite{toth1976research}, we are interested in minimal densities of lattice arrangements with the more refined covering property captured by the $i$-th covering minimum.
In particular, we investigate the following problem.

\begin{problem}\label{probMainProblem}
Find optimal constants $c_{1,n},\ldots,c_{n,n}>0$ and $c_n>0$ that depend only on their indices, such that for any $K\in\K^n$ and $\Lambda\in\L^n$, one has
\begin{align}
\mu_i(K,\Lambda)^n\vol(K)&\geq c_{i,n}\det(\Lambda),\label{eqnMainProblem1}
\end{align}
for $i=1,\ldots,n$, and
\begin{align}
\mu_1(K,\Lambda)\cdot\ldots\cdot\mu_n(K,\Lambda)\vol(K)&\geq c_n\det(\Lambda).\label{eqnMainProblem2}
\end{align}
\end{problem}

Observe that the inequalities in \cref{probMainProblem} are invariant under simultaneous transformations of $K$ and $\Lambda$ by an invertible linear mapping.
Therefore, we usually restrict our attention to the standard lattice $\Lambda=\Z^n$ without loss of generality.
The question~\eqref{eqnMainProblem2} involving the whole sequence of covering minima was already posed by Betke, Henk \& Wills~\cite{betkehenkwills1993successive} and in the case $n=2$ answered by Schnell~\cite{schnell1995a} (cf.~\cref{thmSchnell}).
This inequality can be seen as a dual inequality to Minkowski's second fundamental theorem~\eqref{eqnMinkowskiSecond}.
For convenience we call $\mu_1(K,\Lambda)\cdot\ldots\cdot\mu_n(K,\Lambda)\vol(K)/\det(\Lambda)$ the \emph{covering product} of $K$ (with respect to~$\Lambda$).

Our contributions to \cref{probMainProblem} focus on the one hand on determining meaningful first bounds on the optimal constants $c_{i,n}$ and $c_n$, and on the other hand, on solving it for a particular family of convex bodies.
To be more precise, in \cref{prop_first_bounds_mu_j} we obtain lower bounds of the type~\eqref{eqnMainProblem1} that support the natural guess that $c_{i,n}$ is in order much bigger than $c_{j,n}$ for any $i>j$.
Regarding the covering product, we prove in \cref{thm_mu_prod_weak_bound} that $c_n\geq1/n!$, which is a necessary condition for \eqref{eqnMakaiConjSymmetric}, since the covering minima form a non-decreasing sequence, that is, $\mu_1(K,\Lambda)\leq\ldots\leq\mu_n(K,\Lambda)$.
For the family of \emph{standard unconditional} convex bodies, which are convex bodies that are symmetric with respect to every coordinate hyperplane, we derive the best possible bounds in~\eqref{eqnMainProblem1} and~\eqref{eqnMainProblem2}, and characterize the extremal convex bodies (see \cref{prop_mu_j_vol_uncondis_standard} and \cref{thm_mu_prod_vol_uncond}).
Finally, we argue in \cref{sectCoveringProduct} that for general convex bodies the optimal constant in~\eqref{eqnMainProblem2} is most likely given by $c_n=(n+1)/2^n$, and we exhibit a concrete example that has exponentially smaller covering product than any standard unconditional body.
All these results show that, unlike for Minkowski's inequalities~\eqref{eqnMinkowskiFirst} and~\eqref{eqnMinkowskiSecond}, the problems~\eqref{eqnMainProblem1} and~\eqref{eqnMainProblem2} are independent from each other.

Before we discuss the details of the aforementioned findings, we survey known results and various connections of Makai Jr.'s conjecture to some notoriously difficult problems in convex and discrete geometry.
Moreover, we illustrate the applicability of the polar Minkowski inequality by deriving a variant of a linear form theorem from a known case of \cref{conjMakai}.


%% file: sec-discussion.tex
\section{A review of the literature around Makai Jr.'s conjecture and an application to linear forms}
\label{sectDiscussion}

For the ease of presentation, we mostly restrict the discussion to the case of $o$-symmetric convex bodies in this section.
There are analogous ``non-symmetric'' versions of the relations elaborated on below, which can easily be found in the cited literature.
Specifically, \'{A}lvarez Paiva, Balacheff \& Tzanev~\cite[Sect.~3]{alvarezpaivabalachefftzanev2013isosystolic} provide detailed information for the general case.

There is a strong connection of Makai Jr.'s conjecture to a well-known problem of Mahler on the volume-product.
Still an unsolved problem today, Mahler conjectured in 1939 that for $o$-symmetric convex bodies $K\in\K^n_o$ the \emph{volume-product} $\M(K)=\vol(K)\vol(K^\star)$ is minimized by the cube $C_n=[-1,1]^n$.
In symbols,
\begin{align}
\M(K)\geq\M(C_n)=\frac{4^n}{n!}.\label{eqnMahlerConjecture}
\end{align}
We refer the reader to~\cite{boroczkymakaimeyerreisner2013onthe} for a historical account of the problem, an overview of the state of the art, and references to the original literature on partial results concerning~\eqref{eqnMahlerConjecture} that we mention below.

Now, Makai Jr.~\cite[Thm.~1]{makai1978on} proved the remarkable identity
\begin{align}
\M(K)&=4^{n}\delta(K)\theta_1(K^\star),\quad\text{for}\quad K\in\K^n_o,\label{eqnMakaiMahlerIdentity}
\end{align}
where $\delta(K)=\max\{\delta(K,\Lambda):\Lambda\in\L^n, K+\Lambda\text{ a packing}\}$ denotes the maximum density of a lattice packing of~$K$, and $\theta_i(K)=\min_{\Lambda\in\L^n}\delta(\mu_i(K,\Lambda)K,\Lambda)$ the minimum density of a lattice arrangement of~$K$ that intersects every $(n\!-\!i)$-dimensional affine subspace.
This relation shows that Mahler's conjecture~\eqref{eqnMahlerConjecture} is equivalent to $\delta(K)\theta_1(K^\star)\geq1/n!$, a statement on densities of lattice arrangements.

In view of $\delta(K)\leq1$, for $K\in\K^n_o$, we see that Makai Jr.'s conjecture~\eqref{eqnMakaiConjSymmetric}, which reformulates as $\theta_1(K)\geq1/n!$, is a necessary condition for Mahler's conjecture.
In particular, partial results for the latter problem transfer to the former.
For example, \cref{conjMakai} holds for $n=2$ (cf.~\cite{fejestothmakai1974plates,makai1978on}), and its $o$-symmetric version~\eqref{eqnMakaiConjSymmetric} holds on the family of unconditional convex bodies, ellipsoids, and zonotopes.
Notice that this means that the polar Minkowski inequality~\eqref{eqnPolarMinkowski} holds for the polar bodies of these special classes.

While the exact conjectured lower bound in~\eqref{eqnMahlerConjecture} remains elusive, the asymptotic growth rate of the dimensional constant is well understood.
The strongest result is due to Kuperberg \cite{kuperberg2008from}, who showed that $\M(K)\geq\pi^n/n!$, for any $K\in\K^n_o$.
As a consequence one obtains the following asymptotic estimates in Makai Jr.'s problem.
For any $K\in\K^n$ holds
\begin{align}
\theta_1(K)&\geq\left(\frac{\pi}{8}\right)^n\frac{n+1}{2^nn!},\qquad\text{and}\qquad\theta_1(K)\geq\left(\frac{\pi}{4}\right)^n\frac{1}{n!},\quad\text{if }K\in\K^n_o.\label{eqnAsymptoticMakai}
\end{align}
The first of these inequalities appears in \'{A}lvarez Paiva et al.~\cite[Thm.~II]{alvarezpaivabalachefftzanev2013isosystolic}.
Note also that already Mahler~\cite{mahler1974polar} studied asymptotic estimates of this kind.

Conversely, it turns out that an affirmative answer to Makai Jr.'s conjecture implies good asymptotic results for Mahler's problem, so that the two conjectures~\eqref{eqnMakaiConjSymmetric} and~\eqref{eqnMahlerConjecture} are asymptotically equivalent.
More precisely, one can use the famous \emph{Minkowski-Hlawka theorem} (cf.~\cite[p.~202]{gruberlekker1987geometry}), which is a reverse statement of Minkowski's fundamental theorem~\eqref{eqnMinkowskiFirst}, and obtain the bound $\M(K)\geq2^n/n!$, under the assumption $\theta_1(K)\geq1/n!$.
Details on this relation have been discussed in~\cite{alvarezpaivabalachefftzanev2013isosystolic}.

Finally, we survey the very limited knowledge on the densities $\theta_i(K)$ for particular convex bodies~$K$.
The original problem of Fejes T\'{o}th~\cite{toth1976research} concerns the densities $\theta_i(B_n)$.
Since the volume-product $\M(B_n)$ is known explicitly, Makai Jr.'s identity~\eqref{eqnMakaiMahlerIdentity} shows that determining $\theta_1(B_n)$ is equivalent to determining $\delta(B_n)$.
This is the lattice sphere packing problem which is solved in dimension $n\leq8$ and $n=24$ (see~\cite{fejestoth1983new} and~\cite[Sect.~29]{gruber2007convex}).
On the other hand, the lattice sphere covering problem is exactly the question on $\theta_n(B_n)$, being solved for $n\leq5$ (see~\cite{fejestoth1983new}).
The only known value of $\theta_i(B_n)$ for $i\notin\{1,n\}$ is $\theta_2(B_3)$ due to Bambah \& Woods~\cite{bambahwoods1994on}; see \cref{tblBallDensities}.

\begin{table}[ht]
\begin{tabular}{|>{$}c<{$}|*{3}{>{$}c<{$}|}}
\hline n=2 & \multicolumn{3}{c|}{\begin{tabular}[t]{c|c} $\theta_1(B_2)=\frac{\sqrt{3}\pi}{8}$\hspace{15pt} & \hspace{15pt}$\theta_2(B_2)=\frac{2\pi}{\sqrt{27}}$ \end{tabular}} \\\hline
n=3 & \,\theta_1(B_3)=\frac{\sqrt{2}\pi}{12}\, & \,\theta_2(B_3)=\frac{9\pi}{32}\, & \,\theta_3(B_3)=\frac{5\sqrt{5}\pi}{24}\, \\\hline
\end{tabular}
\caption{Densities of lattice arrangements of $B_n$ in small dimensions.}
 \label{tblBallDensities}
\end{table}
\noindent Since the cube admits a lattice packing that covers the whole space, we have $\delta(C_n)=\theta_n(C_n)=1$ and hence via~\eqref{eqnMakaiMahlerIdentity} and $\M(C_n)=4^n/n!$, we find $\theta_1(C_n^\star)=1/n!$.
Exchanging the roles of $C_n$ and $C_n^\star$ leads to a difficult problem for which only recently the first non-trivial results were proven.
In~\cite{fejestothfodorvigh2015packing} it was shown that there is an absolute constant $c>0$ such that $\delta(C_n^\star)\leq c\cdot 0.8685^n$, and hence $\theta_1(C_n)\geq c \cdot 1.1514^n/n!$.
We think that the probably most managable problem on these densities is the following:
\[\text{Is it true that }\theta_{n-1}(C_n)=1/2,\text{ for any }n\geq2\,?\]
This conjectured value would be realized by the checkerboard lattice $\Lambda_o=\{x\in\Z^n:x_1+\ldots+x_n \equiv 0 \text{ mod }2\}$ that also appears at the end of \cref{sectMuiVolBounds}.

\subsection*{An application to linear forms}

Minkowski successfully applied his fundamental theorem~\eqref{eqnMinkowskiFirst} to questions concerning solutions of inequalities involving linear forms; a benchmark example is his ``linear form theorem'' (cf.~\cite[Cor.~22.2]{gruber2007convex}).
It works best in situations where the volume of the underlying convex body that describes the problem at hand, can be explicitly computed.
It should come as no surprise that an affirmative answer to Makai Jr.'s conjecture~\eqref{eqnMakaiConjSymmetric} would be equally useful to solve questions in which the volume of the polar body can be controlled.
An illustrating example is the following.

\begin{theorem}
Consider $n$ linear homogeneous forms $\ell_i(x)=a_i^\intercal x$, for some $a_1,\ldots,a_n\in\R^n$ of determinant $\det(A)=\det(a_1,\ldots,a_n)\neq0$.
Then, there exists a non-zero integral vector $x\in\Z^n\setminus\{0\}$ such that
\[\sum_{i=1}^n\card{\ell_i(x)}+\card*{\sum_{i=1}^n\ell_i(x)}\leq\left((n+1)!\cdot\card{\det(A)}\right)^{\frac{1}{n}}.\]
For small dimensions one can provide sharp bounds:

If $n=2$, there exists a non-zero integral vector $x\in\Z^n\setminus\{0\}$ such that
\[\card{\ell_1(x)}+\card{\ell_2(x)}+\card*{\ell_1(x)+\ell_2(x)}\leq\frac{4}{\sqrt{3}}\card{\det(A)}^{\frac12}.\]
The forms $\ell_1(x)=x_1-2x_2$, $\ell_2(x)=x_1+x_2$ show that the constant on the right hand side cannot be improved.

If $n=3$, there exists a non-zero integral vector $x\in\Z^n\setminus\{0\}$ such that
\[\card{\ell_1(x)}+\card{\ell_2(x)}+\card{\ell_3(x)}+\card*{\ell_1(x)+\ell_2(x)+\ell_3(x)}\leq6\left(\frac{2}{7}\right)^{\frac23}\card{\det(A)}^{\frac13}.\]
The forms $\ell_1(x)=3x_1+3x_2-4x_3$, $\ell_2(x)=3x_1-4x_2+3x_3$, and $\ell_3(x)=-4x_1+3x_2+3x_3$ show the minimality of the constant.
\end{theorem}
\begin{proof}
Consider the zonotope $Z_n=[-1,1]^n+[-{\bf 1},{\bf 1}]$.
This is an $o$-symmetric convex body whose volume can be computed by dissecting $Z_n$ into $n+1$ parallelepipeds, and which is given by $\vol(Z_n)=(n+1)2^n$ (see~\cite[Ch.~9]{beckrobins2015computing2nd}).
Denoting the \emph{norm function} of $K\in\K^n_o$ by $\|y\|_K=\min\{\lambda\geq0:y\in \lambda K\}$, we have
\[\|y\|_{Z_n^\star}=h_{Z_n}(y)=\sum_{i=1}^n\card{y_i}+\card{y_1+\ldots+y_n},\quad\text{for any}\quad y\in\R^n,\]
where $h_{Z_n}(x)=\max_{y\in Z_n}x^\intercal y$, $x\in\R^n$, is the \emph{support function} of~$Z_n$ (see~\cite{gruber2007convex}).
Therefore, for any $\tau\geq0$, we have
\[\left\{x\in\R^n:\sum_{i=1}^n\card{\ell_i(x)}+\card*{\sum_{i=1}^n\ell_i(x)}\leq\tau\right\}=\tau A^{-1} Z_n^\star.\]
By the formulation of~\eqref{eqnMakaiConjSymmetric} as a polar Minkowski theorem~\eqref{eqnPolarMinkowski} and its validity for polars of zonotopes, we get that $\tau A^{-1} Z_n^\star$ contains a non-zero integral vector $x\in\Z^n$, if
\[\frac{2^n}{n!}\geq\vol\left((\tau A^{-1} Z_n^\star)^\star\right)=\frac{\card{\det(A)}}{\tau^n}\vol(Z_n)=\frac{\card{\det(A)}}{\tau^n}(n+1)2^n.\]
This holds if and only if $\tau^n\geq(n+1)!\cdot\card{\det(A)}$, implying the claim for arbitrary~$n$.

For $n\in\{2,3\}$, the density of a densest lattice packing of $Z_n^\star$ is known, and moreover, we can compute the volume of $Z_n^\star$ explicitly.
These facts enable us to use Minkowski's fundamental theorem instead of~\eqref{eqnPolarMinkowski} and they lead to the sharp bounds stated in the theorem.
By definition of the density $\delta(Z_n^\star)$ one can introduce it as a parameter in~\eqref{eqnMinkowskiFirst} and obtain that $\lambda_1(Z_n^\star,\Z^n)^n\vol(Z_n^\star)\leq2^n\delta(Z_n^\star)$ (cf.~\cite[\S~20.1]{gruberlekker1987geometry}).
Since the density $\delta(Z_n^\star)$ is invariant under invertible linear transformations, this inequality guarantees the existence of a non-zero integral vector in $\tau A^{-1}Z_n^\star$ as long as
\begin{align}
\vol(\tau A^{-1}Z_n^\star)&=\frac{\tau^n}{\card{\det(A)}}\vol(Z_n^\star)\geq\delta(Z_n^\star)2^n.\label{eqnLFsmallDim}
\end{align}
Now, $Z_2^\star$ is a hexagon of volume $\vol(Z_2^\star)=3/4$ which tiles the plane by translations of the lattice $\frac12\Bigg(\begin{array}{cc}1&-2\\ 1&1\end{array}\Bigg)\Z^2$.
Hence, $\delta(Z_2^\star)=1$, and~\eqref{eqnLFsmallDim} gives the condition $\tau^2\geq(16/3)\card{\det(A)}$, from which the claimed bound follows, for $n=2$.
The extremal example can be read off from the lattice that realizes~$\delta(Z_2^\star)$.

In the case $n=3$, we find that $Z_3^\star$ is the cubeoctahedron, which is an Archimedean polytope with $12$ vertices and volume $\vol(Z_3^\star)=5/12$.
A densest lattice packing of $Z_3^\star$ has been computed by Betke \& Henk~\cite[Sect.~5]{betkehenk2000densest}.
After a linear transformation respecting our coordinates of the cubeoctahedron, their result shows that $\delta(Z_3^\star)=45/49$, and this density is realized by the lattice with basis $(1/6)\cdot\left\{(3,3,-4)^\intercal,(3,-4,3)^\intercal,(-4,3,3)^\intercal\right\}$.
Analogously to the case $n=2$, we use this information together with~\eqref{eqnLFsmallDim} in order to obtain the desired bound and an extremal example.
\end{proof}

%% file: sec-mu_i_vol_bounds.tex
\section{Results concerning inequalities of the form~\texorpdfstring{\eqref{eqnMainProblem1}}{(\ref{eqnMainProblem1})}}
\label{sectMuiVolBounds}

For a lattice $\Lambda\in\L^n$ we denote by $\L_i(\Lambda)$ the family of $i$-dimensional \emph{lattice planes} of $\Lambda$, that is, linear subspaces of $\R^n$ that are spanned by vectors of~$\Lambda$.
Kannan \& Lov\'{a}sz~\cite[Rem.~1]{kannanlovasz1988covering} observed that the $i$-th covering minimum admits an equivalent description via projections onto $i$-dimensional lattice planes.
Denoting by $S|L$ the orthogonal projection of $S\subseteq\R^n$ onto a linear subspace $L$, they prove that
\begin{align}
\mu_i(K,\Lambda)&=\max\left\{\mu_i(K|L,\Lambda|L):L\in\L_i(\Lambda)\right\}.\label{eqn_KL_max_mu}
\end{align}
Note that $\Lambda|L$ is a lattice in the subspace $L\in\L_i(\Lambda)$, so that it makes sense to compute the covering radius of $K|L$ with respect to this lattice.

As explained in the introduction, there is no loss of generality to restrict the consideration to the standard lattice~$\Z^n$ in \cref{probMainProblem}.
For the sake of brevity, we write $\mu_i(K)=\mu_i(K,\Z^n)$ in this case, for any $K\in\K^n$ and any $i\in[n]:=\{1,\ldots,n\}$.
We start our discussion with first general bounds of the form~\eqref{eqnMainProblem1}.

\begin{theorem}\label{prop_first_bounds_mu_j}\
\begin{enumerate}[i)]
 \item Let $K\in\K^n_o$. There exists a constant $\flt(n)$ only depending on~$n$ such that, for any $i\in[n]$,
  \[\mu_i(K)^n\vol(K)\geq\frac{i!}{n!}\flt(n)^{-(n-i)}.\]
  In particular, we can choose $\flt(n)=c\,n(1+\log{n})$, for some $c>0$.
 \item Let $K\in\K^n$ be such that $\mu_i(K)K+\Z^n$ contains every $i$-dimensional coordinate hyperplane. Then
\[\mu_i(K)^n\vol(K)\geq\frac{i!^{\frac{n}{i}}}{n!}.\]
\end{enumerate}
\end{theorem}
\begin{proof}
\romannumeral1): First of all, Jarn\'{i}k's inequality~\cite{jarnik1941zwei} (cf.~\cite[Lem.~(2.4)]{kannanlovasz1988covering}) claims that $\lambda_n(K)/2\leq\mu_n(K)$.
In view of the \emph{flatness theorem} there is a constant $\flt(n)$ only depending on~$n$ such that $\mu_n(K)\leq\flt(n)\mu_1(K)$ (cf.~\cite[Thm.~(2.7)]{kannanlovasz1988covering}).
Therefore, $\lambda_n(K)\leq 2\flt(n)\mu_1(K)$, and since $\lambda_i(K)$ and $\mu_i(K)$ form non-decreasing sequences with respect to $i\in[n]$, we get $\lambda_{n-i}(K)\leq2\flt(n)\mu_i(K)$.
This means that there is an $(n\!-\!i)$-dimensional lattice plane $L\in\L_{n-i}(\Z^n)$ and linearly independent vectors $a_1,\ldots,a_{n-i}\in\Z^n\cap L$, such that the crosspolytope $\conv\{\pm a_1,\ldots,\pm a_{n-i}\}$ is contained in the body $2\flt(n)\mu_i(K) K\cap L$.
Thus,
\begin{align}
\vol_{n-i}(K\cap L)&\geq\left(\frac1{2\flt(n)\mu_i(K)}\right)^{n-i}\frac{2^{n-i}}{(n-i)!}\det(\Z^n\cap L),\label{eqnFirstBoundsIntersection}
\end{align}
where, for any Lebesgue-measurable $i$-dimensional set $S\subseteq\R^n$, we denote by $\vol_i(S)$ the volume of $S$ computed in its affine hull.
Let $L^\perp$ be the orthogonal complement of~$L$.
By~\eqref{eqn_KL_max_mu}, we have $\mu_i(K)\geq\mu_i(K|L^\perp,\Z^n|L^\perp)$, and applying inequality~\eqref{eqn_mu_n} to the projected body $K|L^\perp$, we obtain
\begin{align}
\mu_i(K)^i\vol_i(K|L^\perp)&\geq\det(\Z^n|L^\perp).\label{eqnFirstBoundsProjection}
\end{align}
Finally, we utilize an inequality of Rogers \& Shephard~\cite[Thm.~1]{rogersshephard1958convex}, which states that
\begin{align}
\vol_{n-i}(K\cap L)\vol_i(K|L^\perp)\leq\binom{n}{i}\vol(K).\label{eqnRogersShephard}
\end{align}
Putting together~\eqref{eqnFirstBoundsIntersection}, \eqref{eqnFirstBoundsProjection}, and~\eqref{eqnRogersShephard}, and using $\det(\Z^n\cap L)\det(\Z^n|L^\perp)=\det(\Z^n)=1$ (see~\cite[Prop.~1.9.7]{martinet2003perfect}), we arrive at
\begin{align*}
\mu_i(K)^n\vol(K)&\geq\frac{\mu_i(K)^n\vol_i(K|L^\perp)\vol_{n-i}(K\cap L)}{\binom{n}{i}\det(\Z^n|L^\perp)\det(\Z^n\cap L)}\\
&\geq\frac{\mu_i(K)^{n-i}}{\binom{n}{i}}\frac{1}{\left(\flt(n)\mu_i(K)\right)^{n-i}(n-i)!}\geq\frac{i!}{n!}\flt(n)^{-(n-i)}.
\end{align*}
Due to a result of Banaszczyk~\cite{banaszczyk1996inequalities}, we can choose $\flt(n)=c\,n(1+\log{n})$, for some absolute constant $c>0$.

\romannumeral2): For $J\subseteq[n]$, let $L_J=\lin\{e_j:j\in J\}$ be the linear subspace spanned by $e_j$, $j\in J$.
An inequality of Meyer~\cite{meyer1988a} states that, for all $i\in[n]$, we have
\[\left(n!\vol(K)\right)^{\frac{i}{n}}\geq i!\Bigg(\prod_{J \subseteq [n], |J|=i}\vol_i(K\cap L_J)\Bigg)^{\frac1{\binom{n}{i}}}.\]
By assumption, every $i$-dimensional coordinate hyperplane $L_J$, $J \subseteq [n]$, $|J|=i$, is covered by the translates $\mu_i(K)K+\Z^n$.
Therefore, by~\eqref{eqn_mu_n} we have that $\vol_i\left(\mu_i(K)K\cap L_J\right)\geq1$.
Hence, Meyer's inequality gives us
\[\mu_i(K)^n\vol(K)=\vol(\mu_i(K)K)\geq\frac{i!^{\frac{n}{i}}}{n!}.\qedhere\]
\end{proof}

\begin{remark}\label{remRoughBoundsMuI}\
\begin{enumerate}[i)]
 \item For $k\in\N$ a constant, \cref{prop_first_bounds_mu_j}~\romannumeral1) yields a lower bound of the type $\mu_{n-k}(K)^n\vol(K)\geq\frac{c}{n^{2k}(\log{n})^k}$ which dramatically improves upon the bound that follows from the monotonicity of the $\mu_i(K)$ and the known asymptotic bounds~\eqref{eqnAsymptoticMakai} on $\mu_1(K)^n\vol(K)$.
 \item The case $i=1$ in \cref{prop_first_bounds_mu_j}~\romannumeral2) is included in the more general setting that $\mu_1(K)K+\Z^n$ is connected, which has been studied in~\cite{toth1973on,groemer1966zusammen}.
\end{enumerate}
\end{remark}

\noindent Our next goal is to derive sharp lower bounds on $\mu_i(K)^n\vol(K)$ on a particular family of convex bodies $K$.
Before we state our result, we introduce and investigate a family of polytopes that interpolates between the cube and the crosspolytope.
For every $i\in[n]$, let
\[P_{n,i}=\conv\{\pm e_{j_1}\pm\ldots\pm e_{j_i}:1\leq j_1<\ldots<j_i\leq n\}=C_n\cap iC_n^\star.\]
Note that $P_{n,n}=C_n$ and $P_{n,1}=C_n^\star$.
Moreover, $P_{3,2}$ is the cubeoctahedron and $P_{4,2}$ is the $24$-cell.
The facet $P_{n,i}\cap\{x\in\R^n:x_1+\ldots+x_n=i\}$ of~$P_{n,i}$ is known as the \emph{$i$-th hypersimplex} and usually denoted by $\Delta_{n-1}(i)$ (see~\cite[Ex.~0.11]{ziegler1995lectures} for more information and the origin of these interesting polytopes).
We may therefore think of the $P_{n,i}$ as the symmetric cousins of the hypersimplices.
The covering minima and the volume of these special polytopes can be computed explicitly as follows.

\begin{proposition}\label{prop_mu_j_vol_P_ni}\
\begin{enumerate}[i)]
 \item For $i\in[n]$, we have
  \[\mu_j(P_{n,i})=\frac12,\text{ for }j\leq i,\quad\text{and}\quad\mu_j(P_{n,i})=\frac{j}{2i},\text{ for }j>i.\]
  In particular, $\mu_i(C_n)=1/2$ and $\mu_i(C_n^\star)=i/2$, for all $i\in[n]$.
 \item For $i\in[n]$, we have\footnote{Aicke Hinrichs (personal communication) pointed out to us that the volume of $P_{n,i}$ can also be computed via a probabilistic argument based on the Irwin-Hall distribution.}
\[\vol(P_{n,i})=\frac{2^n}{n!}\sum_{k=0}^i(-1)^k\binom{n}{k}(i-k)^n.\]
\end{enumerate}
\end{proposition}
\begin{proof}
\romannumeral1): Let $j\in\{1,\ldots,i\}$ and let $L_j$ be a $j$-dimensional coordinate subspace. Then, $P_{n,i}|L_j=P_{n,i}\cap L_j=C_j$, where we identify $L_j$ with $\R^j$.
Since also $L_j\subseteq\frac12 P_{n,i}+\Z^n$, we get $\mu_j(P_{n,i})=1/2$.
For the case $j\in\{i+1,\ldots,n\}$, we first note that since $P_{n,i}\subseteq C_n$ and $(i/n,\ldots,i/n)$ lies in the boundary of~$P_{n,i}$, we have $\mu_n(P_{n,i})=n/(2i)$.
Using that $P_{n,i}|L_j=P_{j,i}$, for all $j$-dimensional coordinate subspaces $L_j$, this implies $\mu_j(P_{n,i})=j/(2i)$.

\romannumeral2): It is known that (see~\cite{stanley1977eulerian} for instance), for any $k\in[n]$,
\[\vol\Big([0,1]^n\cap\Big\{x\in\R^n:k-1\leq\sum_{i=1}^nx_i\leq k\Big\}\Big)=\frac1{n!}A_{n,k},\]
where $A_{n,k}$ denotes the Eulerian numbers.
Therefore,
\begin{align*}
\vol(P_{n,i})&=2^n\vol\Big([0,1]^n\cap\Big\{x\in\R^n:0\leq\sum_{i=1}^nx_i\leq i\Big\}\Big)=\frac{2^n}{n!}\sum_{k=1}^iA_{n,k},
\end{align*}
which implies the desired formula in view of the identity (cf.~\cite[Ch.~2]{beckrobins2015computing2nd})
\[A_{n,k}=\sum_{j=0}^k(-1)^j\binom{n+1}{j}(k-j)^n\]
and routine algebraic manipulations.
\end{proof}

Now, recall that a convex body $K\in\K^n_o$ is called \emph{standard unconditional} if for every $x\in K$, we have $(\pm x_1,\ldots,\pm x_n)\in K$, that is, $K$ is symmetric with respect to every coordinate hyperplane.
Observe that by construction the polytopes $P_{n,i}$ are standard unconditional.

\begin{theorem}\label{prop_mu_j_vol_uncondis_standard}
Let $K\in\K^n_o$ be standard unconditional and let $i\in[n]$. Then,
\[\mu_i(K)^n\vol(K)\geq\mu_i(P_{n,i})^n\vol(P_{n,i}).\]
Equality holds if and only if $K=\frac1{2\mu_i(K)}P_{n,i}$.
\end{theorem}
\begin{proof}
As before, for $i\in[n]$ and $J \subseteq [n]$, $|J|=i$, we let $L_J=\lin\{e_j:j\in J\}$.
By the unconditionality of $K$, we get $K|L_J=K\cap L_J$, and furthermore $\Z^n|L_J=\Z^n\cap L_J=\Z^i$, where we identify $L_J$ with $\R^i$.
Since the claimed inequality is invariant under scalings of~$K$, we may assume that
\[\mu_i(K)=\max\{\mu_i(K|L,\Z^n|L):L\in\L_i(\Z^n)\}=1.\]
Then, clearly $\mu_i(K|L_J,\Z^n|L_J)\leq1$, which, by the above observations, means that $K\cap L_J+\Z^n\cap L_J$ covers $L_J$, for all $J \subseteq [n]$, $|J|=i$.
The body $K\cap L_J$ is a standard unconditional body in $L_J$.
Therefore, the covering property yields that $\frac12\sum_{j\in J}e_j\in K\cap L_J\subseteq K$, for all $J \subseteq [n]$, $|J|=i$.
By the definition of~$P_{n,i}$, we see that $\frac12 P_{n,i}\subseteq K$, and hence, using~\cref{prop_mu_j_vol_P_ni}~\romannumeral1), $\vol(K)\geq\vol(\frac12 P_{n,i})=\mu_i(P_{n,i})^n\vol(P_{n,i})$ as desired.

Equality holds if and only if $K=\frac12 P_{n,i}$ which implies the claimed equality case characterization because we assumed that $\mu_i(K)=1$.
\end{proof}

Since the cases $i=n$ and $i=1$ of \cref{prop_mu_j_vol_uncondis_standard} correspond exactly to~\eqref{eqn_mu_n} and~\eqref{eqnMakaiConjSymmetric}, respectively, it is tempting to conjecture that the polytopes~$P_{n,i}$ minimize the functional $\mu_i(K)^n\vol(K)$ on the whole class~$\K^n_o$ of $o$-symmetric convex bodies.
However, the following examples show that this is not the case in general:
Consider the \emph{checkerboard lattice}
\[\Lambda_o=\left\{x\in\Z^n:x_1+\ldots+x_n\equiv0\text{ mod }2\right\},\]
which is a sublattice of $\Z^n$ of determinant $\det(\Lambda_o)=2$ (see~\cite[Ch.~4]{martinet2003perfect}).
We leave it to the reader to check that $\frac12C_n+\Lambda_o$ intersects every line and $C_n^\star+\Lambda_o$ intersects every $(n\!-\!2)$-dimensional affine subspace (cf.~\cite[pg.~588]{kannanlovasz1988covering}).
Moreover, we cannot shrink $\frac12C_n$ or $C_n^\star$ in order to maintain the corresponding intersection property, and hence $\mu_{n-1}(C_n,\Lambda_o)=1/2$ and $\mu_2(C_n^\star,\Lambda_o)=1$.
In view of \cref{prop_mu_j_vol_P_ni}, we obtain
\[\mu_{n-1}(P_{n,n-1})^n\vol(P_{n,n-1})=\frac{n!-1}{n!}>\frac12=\mu_{n-1}(C_n,\Lambda_o)^n\frac{\vol(C_n)}{\det(\Lambda_o)},\]
and
\[\mu_2(P_{n,2})^n\vol(P_{n,2})=\frac{2^n-n}{n!}>\frac{2^{n-1}}{n!}=\mu_2(C_n^\star,\Lambda_o)^n\frac{\vol(C_n^\star)}{\det(\Lambda_o)},\]
for every $n\geq3$. 

%% file: sec-product_functional.tex
\section{Results concerning the covering product}
\label{sectCoveringProduct}

In this section, we are interested in the dual version of Minkowski's inequality~\eqref{eqnMinkowskiSecond}, that is, in lower bounds on the covering product~\eqref{eqnMainProblem2} in \cref{probMainProblem}.
For the case of planar convex bodies this was completely solved by Schnell~\cite{schnell1995a}.
His result shows that, unlike \cref{conjMakai}, the covering product does not distinguish between $o$-symmetric and general convex bodies.

\begin{theorem}[Schnell~\cite{schnell1995a}]\label{thmSchnell}
For any $K\in\K^2$, one has
\[\mu_1(K)\mu_2(K)\vol(K)\geq\frac34.\]
Equality holds, up to transformations that do not change the covering product, exactly for one triangle, one parallelogram, one trapezoid, one pentagon, and one hexagon.
\end{theorem}

In the proof of~\cite[Lem.~(2.5)]{kannanlovasz1988covering}, the authors derive the following very useful property.
Because the arguments are somewhat implicitly given in~\cite{kannanlovasz1988covering}, we provide the reader with a short proof.

\begin{lemma}\label{lem_mu_geq_mu_proj}
Let $K\in\K^n$ and let $j\in[n]$. Then,
\[\mu_j(K)\geq\mu_j(K|L,\Z^n|L),\]
for every lattice plane $L\in\L_i(\Z^n)$ of dimension $i\geq j$.
\end{lemma}
\begin{proof}
Let $i\geq j$ and let $L\in\L_i(\Z^n)$.
Assume that $\mu_j(K)=1$, that is, every $(n\!-\!j)$-dimensional affine subspace intersects $K+\Z^n$.
Now, let $M$ be an $(i\!-\!j)$-dimensional affine subspace in~$L$.
Consider the subspace $M'$ that is the preimage of $M$ under the projection onto $L$, in symbols, $M'=M\oplus L^\perp$.
Clearly, $M'$ is an $(n\!-\!j)$-dimensional subspace in $\R^n$.
By assumption it must intersect $K+\Z^n$, and hence $(M'\cap(K+\Z^n))|L=M\cap(K|L+\Z^n|L)\neq\emptyset$.
Therefore, $\mu_j(K)=1\geq\mu_j(K|L,\Z^n|L)$ as desired.
\end{proof}

With the help of this lemma we can give a lower bound on the covering product of an arbitrary convex body $K\in\K^n$, which, in view of the monotonicity of the sequence of covering minima, is a necessary inequality for \cref{conjMakai}~\eqref{eqnMakaiConjSymmetric} to hold.

\begin{theorem}\label{thm_mu_prod_weak_bound}
Let $K\in\K^n$. Then,
\[\mu_1(K)\cdot\ldots\cdot\mu_n(K)\vol(K)\geq\frac{1}{n!}.\]
\end{theorem}
\begin{proof}
Recall that $\D K=K-K$ denotes the difference body of~$K$.
From Jarn\'{i}k's inequality~\cite{jarnik1941zwei,kannanlovasz1988covering} we know that $\mu_n(K)\geq\lambda_n(\D K) \geq \lambda_1(\D K)$.
In particular, there exists a vector $z\in\mu_n(K)\D K\cap\Z^n\setminus\{0\}$.
This means that there are points $x,y\in K$ such that $z=\mu_n(K)(x-y)$.
Hence, putting $L_z=\lin\{z\}$, the line $L=y+L_z$ passing through~$x$ and~$y$ satisfies
\[\vol_1(K\cap L)\geq\|x-y\|=\frac{\|z\|}{\mu_n(K)}\geq\frac{\det(\Z^n\cap L_z)}{\mu_n(K)}.\]
By \cref{lem_mu_geq_mu_proj}, it holds $\mu_j(K)\geq\mu_j(K|L_z^\perp,\Z^n|L_z^\perp)$, for every $j=1,\ldots,n-1$.
Based on these observations and the inequality~\eqref{eqnRogersShephard} of Rogers \& Shephard with respect to the line $L$, we obtain inductively, that
\begin{align*}
\mu_1(K)&\cdot\ldots\cdot\mu_n(K)\vol(K)\\
&\geq\frac{1}{n}\frac{\mu_1(K)\cdot\ldots\cdot\mu_{n-1}(K)\vol_{n-1}(K|L^\perp)}{\det(\Z^n|L_z^\perp)}\frac{\mu_n(K)\vol_1(K\cap L)}{\det(\Z^n\cap L_z)}\\
&\geq\frac{1}{n}\frac{\mu_1(K|L_z^\perp,\Z^n|L_z^\perp)\cdot\ldots\cdot\mu_{n-1}(K|L_z^\perp,\Z^n|L_z^\perp)\vol_{n-1}(K|L_z^\perp)}{\det(\Z^n|L_z^\perp)}\\
&\geq\frac1{n}\frac1{(n-1)!}=\frac1{n!}.
\end{align*}
We also used the identity $\det(\Z^n\cap L_z)\det(\Z^n|L_z^\perp)=\det(\Z^n)=1$ again.
\end{proof}

\begin{remark}
Based on \cref{prop_first_bounds_mu_j}~\romannumeral1) (cf.~\cref{remRoughBoundsMuI}) one can slightly improve the lower bound $1/n!$ in \cref{thm_mu_prod_weak_bound} to roughly $(1/n!)^{(n-k)/n}$, for any $o$-symmetric $K\in\K^n_o$ and any fixed constant $k\in\N$.

This is meaningful because it shows that (for $o$-symmetric convex bodies) Makai Jr.'s conjecture \cref{conjMakai} is independent from \cref{probMainProblem}~\eqref{eqnMainProblem2}.
\end{remark}

The previous bound improves drastically on the family of standard unconditional bodies.
Indeed, we prove a sharp bound on this class and illustrate that there are infinitely many non-equivalent extremal examples.

\begin{theorem}\label{thm_mu_prod_vol_uncond}
Let $K\in\K^n_o$ be a standard unconditional body. Then,
\[\mu_1(K)\cdot\ldots\cdot\mu_n(K)\vol(K)\geq1.\]
Equality holds if and only if $K=\conv\left\{\frac1{2\mu_1(K)}P_{n,1},\ldots,\frac1{2\mu_n(K)}P_{n,n}\right\}$.
\end{theorem}
\begin{proof}
For the sake of brevity, for any $i\in[n]$, we write $\mu_i=\mu_i(K)$ and $\bar P_{n,i}=\frac12 P_{n,i}$.
From the proof of \cref{prop_mu_j_vol_uncondis_standard}, we know that $\bar P_{n,i}\subseteq\mu_i\cdot K$, for $i\in[n]$.
Therefore, the body $Q_K=\conv\left\{\frac1{\mu_1}\bar P_{n,1},\ldots,\frac1{\mu_n}\bar P_{n,n}\right\}$ is contained in $K$ and it suffices to prove that $\vol(Q_K)=\left(\mu_1\cdot\ldots\cdot\mu_n\right)^{-1}$.

To this end, we first observe that if for some $i,j\in[n], i\neq j$, we would have $\frac1{\mu_i}\bar P_{n,i}\subseteq\inter\big(\frac1{\mu_j}\bar P_{n,j}\big)\subseteq\inter(K)$, then by compactness of the involved bodies, there would be an $\varepsilon>0$ such that $\frac1{\mu_i-\varepsilon}\bar P_{n,i}\subseteq K$, contradicting the minimality of $\mu_i$.
This means, that for any $i\in[n]$, all the vertices of $\frac1{\mu_i}\bar P_{n,i}$ lie in the boundary of~$Q_K$.
Using the vertex description $P_{n,i}=\conv\{\pm e_{j_1}\pm\ldots\pm e_{j_i}:1\leq j_1<\ldots<j_i\leq n\}$, we see that the rays emanating from the origin and passing through the vertices of the polytopes $P_{n,i}$, $i\in[n]$, intersect a facet of the cube $C_n=P_{n,n}$ precisely in the barycenters of its faces, and hence induce a barycentric subdivision of that facet (see~\cite{ziegler1995lectures} for basic notions on convex polytopes).
More precisely, a ray passing through a vertex of $P_{n,i}$ meets a facet of $P_{n,n}$ in the barycenter of one of its $(n-i)$-dimensional faces.
By the symmetry of our construction and the involved bodies, all simplices in the barycentric subdivision have the same volume.
Moreover, there are $(n-1)!\,2^{n-1}$ simplices in the barycentric subdivision of each facet of the cube as can be seen, for example, by counting the number of simplices containing a fixed vertex.
All this shows, that the boundary of $Q_K$ can be triangulated into $n!\,2^n$ simplices all being congruent to $S=\conv\left\{\frac1{2\mu_1}e_1,\frac1{2\mu_2}(e_1+e_2),\ldots,\frac1{2\mu_n}(e_1+\ldots+e_n)\right\}$, and thus
\[\vol(Q_K)=n!\,2^n\vol(\conv\{0,S\})=\frac1{\mu_1\cdot\ldots\cdot\mu_n},\]
as desired.

Since the only step where we could lose some volume is in the inclusion $Q_K\subseteq K$, we immediately see that equality holds if and only if $K=Q_K$.
\end{proof}

\begin{corollary}\label{cor_mu_prod_vol_uncond_intersections}
Let $S\in\K^n$ be such that $S=K\cap\R^n_{\geq0}$ for some standard unconditional body $K\in\K^n_o$.
Then, \[\mu_1(S)\cdot\ldots\cdot\mu_n(S)\vol(S)\geq1,\]
with equality if and only if $S=\conv\left\{\frac1{\mu_1(S)}P_{n,1},\ldots,\frac1{\mu_n(S)}P_{n,n}\right\}\cap\R^n_{\geq0}$.
\end{corollary}
\begin{proof}
By the unconditionality of~$K$, we get that $2^n\vol(S)=\vol(K)$.
In order for a dilate of~$S$ to induce a lattice covering of the whole space it needs to contain the unit cube $[0,1]^n$.
On the other hand, for $K$ it suffices to cover half of it, that is, $[0,1/2]^n$.
By similar considerations as in the proof of \cref{prop_mu_j_vol_P_ni} this shows that $\mu_i(S)=2\mu_i(K)$, for all $i\in[n]$.
In view of these identities, the claimed inequality together with its equality case characterization follow from \cref{thm_mu_prod_vol_uncond}.
\end{proof}

The obtained constants in \cref{thm_mu_prod_weak_bound} and \cref{thm_mu_prod_vol_uncond} are of course far off from each other.
We conjecture that the truth lies somehow in between and that the biggest possible lower bound on the covering product decreases exponentially with the dimension.

\begin{conjecture}\label{conj_mu_prod}
For every $K\in\K^n$, we have
\[\mu_1(K)\cdot\ldots\cdot\mu_n(K)\vol(K)\geq\frac{n+1}{2^n},\]
with equality, for example, for $K=T_n=\conv\{e_1,\ldots,e_n,-{\bf 1}\}$.
\end{conjecture}

For $n=2$, the above conjecture reduces to Schnell's inequality (\cref{thmSchnell}).
In analogy to Schnell's result, we expect that there are many more extremal examples minimizing the covering product, and that there are also $o$-symmetric bodies among them.
These extremal examples, even for $n=3$, can be seen as interesting variants of the so-called \emph{parallelohedra}, which are convex bodies that are extremal in~\eqref{eqn_mu_n} (cf.~\cite[Sect.~32]{gruber2007convex}).

It turns out that even for the simplex~$T_n$, the determination of the whole sequence of covering minima is a highly non-trivial task.
A general upper bound on the covering minima for simplices with integral vertices is the following.
For any full-dimensional simplex $T=\conv\{v_0,v_1,\ldots,v_n\}$, with $v_j\in\Z^n$, we have
\begin{align}
\mu_i(T)\leq \mu_i(S_{\bf 1})=i,\quad\text{for all }i\in[n],\label{eqnCoverMinsSLatticeimplices}
\end{align}
where $S_{\bf 1}=\conv\{0,e_1,\ldots,e_n\}$ is the \emph{standard simplex}.
Indeed, if we let $V\in\Z^{n\times n}$ be the matrix with columns $v_i-v_0$, $i\in[n]$, then $T=VS_{\bf 1}+v_0$.
Moreover, $V\Z^n\subseteq\Z^n$, because $V$ has only integral entries, and thus $VS_{\bf 1}+V\Z^n\subseteq VS_{\bf 1}+\Z^n$.
The definition of the covering minima then implies that $\mu_i(S_{\bf 1})=\mu_i(VS_{\bf 1},V\Z^n)\geq\mu_i(VS_{\bf 1})=\mu_i(T)$.
The fact that $\mu_i(S_{\bf 1})=i$, for all $i\in[n]$, is easy to see, for instance, by similar arguments as in the proof of \cref{prop_mu_j_vol_P_ni}.

In the sequel, we concentrate on determining the exact value of $\mu_n(T_n)$ in arbitrary dimension.
This will be enough to show that the covering product of $T_n$ decreases exponentially with~$n$, and thus motivates \cref{conj_mu_prod}.

\begin{proposition}\label{prop_mu_j_simplex}
Let $T_n=\conv\{e_1,\ldots,e_n,-{\bf 1}\}$. Then,
\begin{enumerate}[i)]
 \item $\mu_n(T_n)=\frac{n}{2}$, and
 \item $\mu_1(T_n)\cdot\ldots\cdot\mu_n(T_n)\vol(T_n)\leq\frac{n+1}{\left(16/15\right)^{7n/15}}\approx\frac{n+1}{1.03057^n}$, for every $n\geq6$.
\end{enumerate}
\end{proposition}

\begin{remark}
For the $i$-dimensional coordinate subspace $L_i=\{e_1,\ldots,e_i\}$ holds $T_n|L_i=T_i\times\{0\}^{n-i}$ and $\Z^n|L_i=\Z^i\times\{0\}^{n-i}$.
Hence by~\eqref{eqn_KL_max_mu}, we have $\mu_i(T_n)\geq\mu_i(T_i)=i/2$.
We conjecture that this is actually an identity for every for $i\in[n]$, which, in view of $\vol(T_n)=(n+1)/n!$, explains the claimed lower bound in \cref{conj_mu_prod}.
\end{remark}

The proof of \cref{prop_mu_j_simplex} needs a bit of preparation.
The covering radius of simplices has appeared in various contexts in the literature.
For instance, a celebrated result of Kannan~\cite{kannan1992lattice} establishes a relation to the Frobenius coin problem, and more recently, Marklof \& Str\"ombergsson~\cite{marklofstrom2013diameters} (cf.~\cite{aliev2015onthe}) made a connection to diameters of so-called quotient lattice graphs.
The latter interpretation suits our purposes well, so we introduce the necessary notation following~\cite{marklofstrom2013diameters}.
We use basic concepts from graph theory, for which the reader may consult the textbook of Diestel~\cite{diestel2010graph}.

Let $\LG_n^+$ be the \emph{standard lattice graph}, that is, the directed graph with vertex set $\Z^n$ and a directed edge $(x,x+e_i)$, for every $x\in\Z^n$ and $i\in[n]$.
For a sublattice $\Lambda\subseteq\Z^n$ the \emph{quotient lattice graph} $\LG_n^+/\Lambda$ is defined as the directed graph with vertex set $\Z^n/\Lambda$ and edges $(x+\Lambda,x+e_i+\Lambda)$, for $x\in\Z^n$ and $i\in[n]$.
For fixed $v\in\R^n_{>0}$, a distance in $\LG_n^+/\Lambda$ is defined by
\[\dist_v(x+\Lambda,y+\Lambda)=\min_{z\in(y-x+\Lambda)\cap\Z^n_{\geq0}}v^\intercal z,\quad\text{for any}\quad x,y\in\Z^n.\]
In other words, $\dist_v(x+\Lambda,y+\Lambda)$ is the length of the shortest directed path from $x+\Lambda$ to $y+\Lambda$ in $\LG_n^+/\Lambda$, where additionally each edge $(x+\Lambda,x+e_i+\Lambda)$ is given the weight~$v_i$, for $i\in[n]$.
The diameter of $\LG_n^+/\Lambda$ (with respect to~$v$) can now be defined as
\[\diam_v(\LG_n^+/\Lambda)=\max_{y\in\Z^n/\Lambda}\dist_v(0+\Lambda,y+\Lambda).\]
Note that, by definition, the distance $\dist_v$ is translation invariant, that is, $\dist_v(x+w+\Lambda,y+w+\Lambda)=\dist_v(x+\Lambda,y+\Lambda)$, for all $x,y,w\in\Z^n$, and hence it suffices to consider paths starting from the vertex $0+\Lambda$ in the definition of $\diam_v(\LG_n^+/\Lambda)$.
Finally, we define the simplex $S_v=\left\{x\in\R^n_{\geq0}:v^\intercal x\leq 1\right\}$.

\begin{theorem}[{\cite[Sect.~2]{marklofstrom2013diameters}}]
  \label{thmDiameterCoveringRadius}
Let $v\in\R^n_{>0}$ and let $\Lambda\subseteq\Z^n$ be a sublattice.
Then,
\[\mu_n(S_v,\Lambda)=\diam_v(\LG_n^+/\Lambda)+v_1+\ldots+v_n.\]
\end{theorem}

Before we can proceed to prove \cref{prop_mu_j_simplex}, we need an auxiliary statement from elementary number theory.
For $k\in\Z$ and $m\in\N$, we let $[k]_m$ be the representative of~$k$ modulo $m$ that lies in $\{0,1,\ldots,m-1\}$.
More precisely, $[k]_m=k-m\lfloor k/m \rfloor$.

\begin{lemma}\label{lemNumberTheory}
For $w\in\Z^{n-1}$ and $r\in\Z$, let $\sigma_w(r)=\sum_{i=1}^{n-1}[w_i+r]_{n+1}$.
\begin{enumerate}[i)]
 \item If $n$ is even, then $|\{\sigma_w(r):r\in\{0,1,\ldots,n\}\}|=n+1$.
 \item If $n$ is odd, then no three of the numbers $\sigma_w(r)$, $r\in\{0,1,\ldots,n\}$, are pairwise equal, and for $r,r'\in\Z$, the difference $|\sigma_w(r)-\sigma_w(r')|$ is even.
\end{enumerate}
\end{lemma}
\begin{proof}
For $j\in\Z$, denote $s_w(j)=|\{i\in[n-1]:[w_i+j]_{n+1}=0\}|$.
With this notation, we find that
\begin{align*}
\sigma_w(0)&=\sum_{i=1}^{n-1}[w_i]_{n+1}=\sum_{i=1}^{n-1}[w_i+1]_{n+1}-(n-1-s_w(1))+ns_w(1)\\
&=\sigma_w(1)+(n+1)s_w(1)-(n-1).
\end{align*}
Therefore, we have
\begin{align}
\sigma_w(k)=\sigma_w(0)-(n+1)\sum_{j=1}^ks_w(j)+k(n-1),\quad\textrm{for}\quad k=1,\ldots,n.\label{eqn_sigma_w_k}
\end{align}
This implies that, for $k,\ell\in\{0,1,\ldots,n\}$ with $k<\ell$, we have $\sigma_w(k)=\sigma_w(\ell)$ if and only if
\begin{align}
(\ell-k)(n-1)=(n+1)\sum_{j=k+1}^\ell s_w(j).\label{eqn_sigma_w_k_equals_l}
\end{align}
This number lies in $\{1,\ldots,n(n-1)\}$ and is a common multiple of $n-1$ and $n+1$.
Now, if $n$ is even, then $\gcd(n-1,n+1)=1$, and hence there is no $k<\ell$ satisfying~\eqref{eqn_sigma_w_k_equals_l}, which proves \romannumeral1).

So, let $n$ be odd.
Then $\gcd(n-1,n+1)=2$, and hence $(n-1)(n+1)/2$ is the only common multiple of $n-1$ and $n+1$ in $\{1,\ldots,n(n-1)\}$.
Assume that there are $k,l,m\in\{0,1,\ldots,n\}$ with $k<\ell<m$ and $\sigma_w(k)=\sigma_w(\ell)=\sigma_w(m)$.
It follows that $\ell-k=m-k=m-\ell=(n+1)/2$ and hence $k=\ell=m$.
This contradiction proves the first claim of~\romannumeral2).
Moreover, $n-1$ and $n+1$ are even so that by~\eqref{eqn_sigma_w_k}, we see that the difference between any $\sigma_w(r)$ and $\sigma_w(r')$ is an even number.
\end{proof}

\begin{proof}[Proof of \cref{prop_mu_j_simplex}]
\romannumeral1): For the computation of the covering radius $\mu_n(T_n)$ it turns out to be more convenient to transform $T_n$ to a multiple of the standard simplex $S_{\bf 1}=\conv\{0,e_1,\ldots,e_n\}$.
To this end, let $A=(a_{ij})\in\Z^{n\times n}$ be the matrix with entries $a_{ii}=n$, for $i\in[n]$, and $a_{ij}=-1$, for $i,j\in[n]$ with $i\neq j$.
Then, we have $AT_n=(n+1)S_{\bf 1}-{\bf 1}$ and writing $\Lambda_n=A\Z^n$, we thus get
\[\mu_n(T_n)=\mu_n(S_{\bf 1},\Lambda_n)/(n+1).\]
Based on this identity and \cref{thmDiameterCoveringRadius}, we infer that
\begin{align}
\mu_n(T_n)&=\frac{n}{2}\quad\text{ if and only if }\quad\diam_{\bf 1}(\LG_n^+/\Lambda_n)=\binom{n}{2}.\label{eqnCovRadLatGraph}
\end{align}
The sublattice $\Lambda_n\subseteq\Z^n$ has a nice structure.
For instance, we have
\[\Lambda_n=\bigcup_{i=0}^n\left(i\cdot{\bf 1}+(n+1)\Z^n\right)\quad\textrm{and}\quad\det(\Lambda_n)=(n+1)^{n-1}.\]
Considering the fundamental cell of $\Lambda_n$ that has vertices $\{0,n+1\}^{n-1}\times\{0\}$ and ${\bf 1}+\{0,n+1\}^{n-1}\times\{0\}$, we see that the quotient lattice graph $\LG_n^+/\Lambda_n$ can be described as follows.
Its vertex set is given by $\{0,1,\ldots,n\}^{n-1}$ and $(x,y)$ is a directed edge if and only if $y-x\in\{e_1,\ldots,e_{n-1},-{\bf 1}\}$ modulo $n+1$ (see \cref{figQuotientLatticeGraphDim3}).
Notice that $e_1,\ldots,e_{n-1}$ and ${\bf 1}$ are points in $\R^{n-1}$ here.
We now prove that the diameter of $\LG_n^+/\Lambda_n$ is as stated in~\eqref{eqnCovRadLatGraph}.

\begin{figure}[ht]
\includegraphics{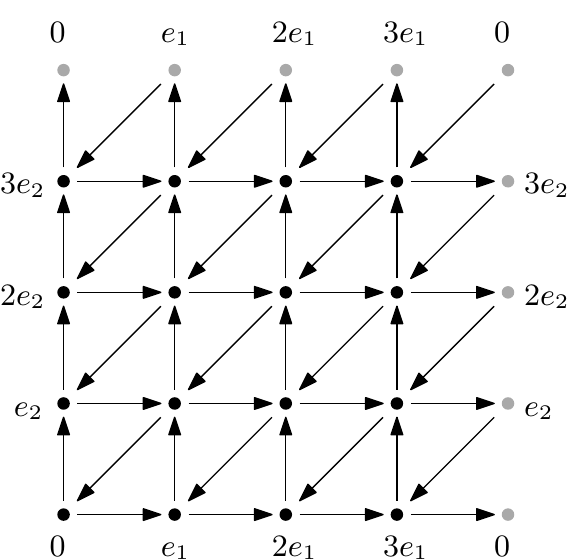}
\caption{The quotient lattice graph $\LG_3^+/\Lambda_3$.}
\label{figQuotientLatticeGraphDim3}
\end{figure}

We first argue that $\diam_{\bf 1}(\LG_n^+/\Lambda_n)\leq\binom{n}{2}$.
For this to be true we need to show that there is a directed path in $\LG_n^+/\Lambda_n$ from $0+\Lambda_n$ to $w+\Lambda_n$ of length at most $\binom{n}2$, for every $w\in\{0,1,\ldots,n\}^{n-1}$.
By the definition of $\LG_n^+/\Lambda_n$ this amounts to finding a representation
\begin{align}
w&=r_1e_1+\ldots+r_{n-1}e_{n-1}-r_n{\bf 1},\label{eqnRepresentation}
\end{align}
for some $r_1,\ldots,r_n\in\Z$ such that $\sum_{i=1}^n[r_i]_{n+1}\leq\binom{n}2$.
Observe that once we fix $r_n\in\{0,1,\ldots,n\}$, we have $r_i=w_i+r_n$, for $i\in[n-1]$.
Hence, there exists a desired representation of~$w$ if and only if there is an $r\in\{0,1,\ldots,n\}$ such that
\begin{align}
\sigma_w(r)&=\sum_{i=1}^{n-1}[w_i+r]_{n+1}\leq\binom{n}{2}-r.\label{eqnConditionDiameter}
\end{align}
For the sake of contradiction, we assume that for every $r\in\{0,1,\ldots,n\}$ the inequality~\eqref{eqnConditionDiameter} does not hold.
Notice that, by definition of $[k]_m$, we have
\begin{align}
\sum_{r=0}^n\sigma_w(r)&=\sum_{i=1}^{n-1}\sum_{r=0}^n[w_i+r]_{n+1}=\sum_{i=1}^{n-1}\sum_{j=0}^n j=(n+1)\binom{n}2.\label{eqn_sum_sigma_w_l}
\end{align}
Now, in the case that $n$ is even, \cref{lemNumberTheory}~\romannumeral1) shows that the numbers $\sigma_w(r)$, $r\in\{0,1,\ldots,n\}$, are pairwise different, and hence by our assumption
\begin{align*}
\sum_{r=0}^n\sigma_w(r)&>\sum_{r=0}^n\left(\binom{n}{2}-r\right)+\sum_{r=0}^n r=(n+1)\binom{n}2,
\end{align*}
contradicting~\eqref{eqn_sum_sigma_w_l}.
The case that $n$ is odd is similar.
By \cref{lemNumberTheory}~\romannumeral2), no three of the numbers $\sigma_w(r)$, $r\in\{0,1,\ldots,n\}$, are pairwise equal, and moreover, any two different of these numbers differ by at least two.
Therefore,
\begin{align*}
\sum_{r=0}^n\sigma_w(r)&\geq\sum_{r=0}^n\left(\binom{n}{2}-r+1\right)+2\sum_{j=0}^{\frac{n-1}{2}}(2j)=(n+1)\binom{n}2+\frac12(n+1),
\end{align*}
contradicting~\eqref{eqn_sum_sigma_w_l} again.
In conclusion, there must be an $r\in\{0,1,\ldots,n\}$ satisfying~\eqref{eqnConditionDiameter}, and hence the diameter of $\LG_n^+/\Lambda_n$ is at most $\binom{n}2$ as claimed.

In order to see that this bound is best possible, we consider the vertex $w=(2,3,\ldots,n)$ of $\LG_n^+/\Lambda_n$.
In any representation of~$w$ of the form~\eqref{eqnRepresentation} with $r=r_n\in\{0,1,\ldots,n\}$, we have $r_i=w_i+r=i+1+r$, for $i\in[n-1]$.
Therefore, if $r=n$, then $\sum_{i=1}^{n-1}[r_i]_{n+1}=\sum_{i=1}^{n-1}i=\binom{n}2>\binom{n}2-r$, and if $r<n$, then
\begin{align*}
\sum_{i=1}^{n-1}[r_i]_{n+1}&=\sum_{i=1}^{n-r-1}(i+1+r)+\sum_{i=n-r}^{n-1}(i+r-n)\\
&=\binom{n}2-r+(n-r-1)\geq\binom{n}2-r,
\end{align*}
with equality only for $r=n-1$.
This means, that there is no path from $0+\Lambda_n$ to $w+\Lambda_n$ of length less than $\binom{n}2$, and thus $\diam_{\bf 1}(\LG_n^+/\Lambda_n)=\binom{n}{2}$.
In view of~\eqref{eqnCovRadLatGraph}, we have thus proven that $\mu_n(T_n)=n/2$.

\romannumeral2): The exponential upper bound on the covering product of $T_n$ can be derived from $\mu_n(T_n)\leq n/2$ only.
Indeed, let $i\in[n]$ be fixed and let $p=i/n\in(0,1]$.
Then, \[\mu_i(T_n)\leq\mu_n(T_n)=\frac{n}{2}= i\cdot\frac{1}{2p}.\]
Hence, for any $j\geq i$, putting $p'=j/n$, we have that $\mu_j(T_n)\leq j\cdot\frac{1}{2p'}\leq j\cdot\frac{1}{2p}$.
Conclusively, using the bound $\mu_j(T_n)\leq j$, for $j\leq i-1$ (see~\eqref{eqnCoverMinsSLatticeimplices}), we get
\begin{align}
\mu_1(T_n)\cdot\ldots\cdot\mu_n(T_n)\vol(T_n)&\leq1\cdot2\cdot\ldots\cdot(np-1)\frac{np}{2p}\cdot\ldots\cdot\frac{n}{2p}\cdot\frac{n+1}{n!}\nonumber\\
&=\frac{n+1}{(2p)^{n(1-p)+1}}\leq\frac{n+1}{\left((2p)^{1-p}\right)^n}.\label{eqnProdFunctTn}
\end{align}
The function $p\mapsto(2p)^{1-p}$ has its maximum in $(0,1]$ attained at a value $q>0.7273$ and is monotonically increasing in the interval $(0,q]$.
For $n\geq6$, we choose $i=\lfloor 7n/10 \rfloor$, which gives $8/15\leq7/10-1/n<p\leq7/10$.
Therefore, we can plug $p=8/15$ into the bound~\eqref{eqnProdFunctTn}, implying the desired estimate.
\end{proof}

%% file: covering_minima_volume.bbl
\providecommand{\bysame}{\leavevmode\hbox to3em{\hrulefill}\thinspace}
\providecommand{\MR}{\relax\ifhmode\unskip\space\fi MR }
\providecommand{\MRhref}[2]{%
  \href{http://www.ams.org/mathscinet-getitem?mr=#1}{#2}
}
\providecommand{\href}[2]{#2}
\begin{thebibliography}{10}

\bibitem{aliev2015onthe}
Iskander Aliev, \emph{On the lattice programming gap of the group problems},
  Oper. Res. Lett. \textbf{43} (2015), no.~2, 199--202.

\bibitem{alvarezpaivabalachefftzanev2013isosystolic}
Juan-Carlos {\'A}lvarez~Paiva, Florent Balacheff, and Kroum Tzanev,
  \emph{Isosystolic inequalities for optical hypersurfaces},
  \url{http://arxiv.org/abs/1308.5522}, 2013.

\bibitem{bambahwoods1994on}
Ram~Prakash Bambah and Alan~C. Woods, \emph{On a problem of {G}. {F}ejes
  {T}\'oth}, Proc. Indian Acad. Sci. Math. Sci. \textbf{104} (1994), no.~1,
  137--156.

\bibitem{banaszczyk1996inequalities}
Wojciech Banaszczyk, \emph{Inequalities for convex bodies and polar reciprocal
  lattices in {$\mathbf R^n$}. {II}. {A}pplication of {$K$}-convexity},
  Discrete Comput. Geom. \textbf{16} (1996), no.~3, 305--311.

\bibitem{beckrobins2015computing2nd}
Matthias Beck and Sinai Robins, \emph{Computing the continuous discretely},
  second ed., Undergraduate Texts in Mathematics, Springer, New York, 2015,
  Integer-point enumeration in polyhedra, With illustrations by David Austin.

\bibitem{betkehenk2000densest}
Ulrich Betke and Martin Henk, \emph{Densest lattice packings of 3-polytopes},
  Comput. Geom. \textbf{16} (2000), no.~3, 157--186.

\bibitem{betkehenkwills1993successive}
Ulrich Betke, Martin Henk, and J{\"o}rg~M. Wills, \emph{Successive-minima-type
  inequalities}, Discrete Comput. Geom. \textbf{9} (1993), no.~2, 165--175.

\bibitem{boroczkymakaimeyerreisner2013onthe}
K\'{a}roly B{\"o}r{\"o}czky, Jr., Endre Makai, Jr., Mathieu Meyer, and Shlomo
  Reisner, \emph{On the volume product of planar polar convex bodies---lower
  estimates with stability}, Studia Sci. Math. Hungar. \textbf{50} (2013),
  no.~2, 159--198.

\bibitem{diestel2010graph}
Reinhard Diestel, \emph{Graph theory}, fourth ed., Graduate Texts in
  Mathematics, vol. 173, Springer, Heidelberg, 2010.

\bibitem{toth1976research}
G{\'a}bor Fejes~T{\'o}th, \emph{Research problem 18}, Period. Math. Hungar.
  \textbf{7} (1976), no.~1, 89--90.

\bibitem{fejestoth1983new}
\bysame, \emph{New results in the theory of packing and covering}, Convexity
  and its applications, Birkh\"auser, Basel, 1983, pp.~318--359.

\bibitem{fejestothfodorvigh2015packing}
G{\'a}bor Fejes~T{\'o}th, Ferenc Fodor, and Viktor V{\'{\i}}gh, \emph{The
  packing density of the {$n$}-dimensional cross-polytope}, Discrete Comput.
  Geom. \textbf{54} (2015), no.~1, 182--194.

\bibitem{toth1973on}
L{\'a}szl{\'o} Fejes~T{\'o}th, \emph{On the density of a connected lattice of
  convex bodies}, Acta Math. Acad. Sci. Hungar. \textbf{24} (1973), 373--376.

\bibitem{fejestothmakai1974plates}
L{\'a}szl{\'o} Fejes~T{\'o}th and Endre Makai, Jr., \emph{On the thinnest
  non-separable lattice of convex plates}, Stud. Sci. Math. Hungar. \textbf{9}
  (1974), 191--193.

\bibitem{groemer1966zusammen}
Helmut Groemer, \emph{Zusammenh\"angende {L}agerungen konvexer {K}\"orper},
  Math. Z. \textbf{94} (1966), 66--78.

\bibitem{gruber2007convex}
Peter~M. Gruber, \emph{Convex and discrete geometry}, Grundlehren der
  Mathematischen Wissenschaften [Fundamental Principles of Mathematical
  Sciences], vol. 336, Springer-Verlag, Berlin, 2007.

\bibitem{gruberlekker1987geometry}
Peter~M. Gruber and Cornelis~G. Lekkerkerker, \emph{Geometry of numbers},
  second ed., North-Holland Mathematical Library, vol.~37, North-Holland
  Publishing Co., Amsterdam, 1987.

\bibitem{jarnik1941zwei}
Vojt{\v e}ch Jarn{\'i}k, \emph{Zwei {B}emerkungen zur {G}eometrie der
  {Z}ahlen}, Vestn{\'i}k Kr\'a\-lovsk\'e \v Cesk\'e Spole\v cnosti Nauk (1941),
  12 pp.

\bibitem{kannan1992lattice}
Ravi Kannan, \emph{Lattice translates of a polytope and the {F}robenius
  problem}, Combinatorica \textbf{12} (1992), no.~2, 161--177.

\bibitem{kannanlovasz1988covering}
Ravi Kannan and L{\'a}szl{\'o} Lov{\'a}sz, \emph{Covering minima and
  lattice-point-free convex bodies}, Ann. of Math. (2) \textbf{128} (1988),
  no.~3, 577--602.

\bibitem{kuperberg2008from}
Greg Kuperberg, \emph{From the {M}ahler conjecture to {G}auss linking
  integrals}, Geom. Funct. Anal. \textbf{18} (2008), no.~3, 870--892.

\bibitem{mahler1974polar}
Kurt Mahler, \emph{Polar analogues of two theorems by {M}inkowski}, Bull.
  Austral. Math. Soc. \textbf{11} (1974), 121--129.

\bibitem{makai1978on}
Endre Makai, Jr., \emph{On the thinnest nonseparable lattice of convex bodies},
  Studia Sci. Math. Hungar. \textbf{13} (1978), no.~1-2, 19--27 (1981).

\bibitem{malikiosis2010adiscrete}
Romanos-Diogenes Malikiosis, \emph{A discrete analogue for {M}inkowski's second
  theorem on successive minima}, Adv. Geom. \textbf{12} (2012), no.~2,
  365--380.

\bibitem{marklofstrom2013diameters}
Jens Marklof and Andreas Str\"{o}mbergsson, \emph{Diameter of random circulant
  graphs}, Combinatorica \textbf{33} (2013), no.~4, 429--466.

\bibitem{martinet2003perfect}
Jacques Martinet, \emph{Perfect lattices in {E}uclidean spaces}, Grundlehren
  der Mathematischen Wissenschaften [Fundamental Principles of Mathematical
  Sciences], vol. 327, Springer-Verlag, Berlin, 2003.

\bibitem{meyer1988a}
Mathieu Meyer, \emph{A volume inequality concerning sections of convex sets},
  Bull. London Math. Soc. \textbf{20} (1988), 151--155.

\bibitem{minkowski1896geometrie}
Hermann Minkowski, \emph{Geometrie der {Z}ahlen}, Bibliotheca Mathematica
  Teubneriana, Teubner, Leipzig-Berlin, 1896, reprinted by Johnson Reprint
  Corp., New York, 1968.

\bibitem{rogersshephard1958convex}
C.~Ambrose Rogers and Geoffrey~C. Shephard, \emph{Convex bodies associated with
  a given convex body}, J. London Math. Soc. (2) \textbf{33} (1958), 270--281.

\bibitem{schnell1995a}
Uwe Schnell, \emph{A {M}inkowski-type theorem for covering minima in the
  plane}, Geom. Dedicata \textbf{55} (1995), no.~3, 247--255.

\bibitem{stanley1977eulerian}
Richard Stanley, \emph{Eulerian partitions of a unit hypercube}, Higher
  Combinatorics (M. Aigner, ed.), Reidel, Dordrecht/Boston (1977), p.~49.

\bibitem{ziegler1995lectures}
G\"unter~M. Ziegler, \emph{Lectures on {P}olytopes}, Graduate Texts in
  Mathematics, vol. 152, Springer-Verlag, 1995.

\end{thebibliography}
